\documentclass[12pt]{article}
\usepackage{amsmath, amscd, amssymb, latexsym, epsfig, color, amsthm, mathtools, tikz-cd, array, adjustbox, bbm, tikz, enumerate, relsize}
\makeatletter
\newtheorem*{rep@theorem}{\rep@title}
\newcommand{\newreptheorem}[2]{%
	\newenvironment{rep#1}[1]{%
		\def\rep@title{#2 \ref{##1}}%
		\begin{rep@theorem}}%
		{\end{rep@theorem}}}
\makeatother

\numberwithin{equation}{section}
\newreptheorem{theorem}{Theorem}
\newtheorem{theorem}{Theorem}[section]
\newtheorem{proposition}[theorem]{Proposition}

\newtheorem{corollary}[theorem]{Corollary}

\newtheorem{lemma}[theorem]{Lemma}

\theoremstyle{definition}

\DeclareMathOperator\lk{\mathrm{lk}}

\DeclareMathOperator\rk{\mathrm{rk}}
\DeclareMathOperator\cost{\mathrm{cost}}

\DeclareMathOperator{\supp}{\mathrm{supp}}

\newcommand{\field}{\mathbbm{k}}

\newcommand{\ZZ}{{\mathbb Z}}

\newcommand{\mub}{\mathcal{M}}

\newcommand{\mideal}{\ensuremath{\mathfrak{m}}}

\newcommand{\Hom}{\ensuremath{\mathrm{Hom}}\hspace{1pt}}
\newcommand{\Ext}{\ensuremath{\mathrm{Ext}}\hspace{1pt}}

\title{Ext and local cohomology modules of face rings of simplicial posets}
\author{Connor Sawaske\\
\small Department of Mathematics\\[-0.8ex]
\small University of Washington\\[-0.8ex]
\small Seattle, WA 98195-4350, USA\\[-0.8ex]
\small \texttt{sawaske@math.washington.edu}
}

\begin{document}
\maketitle

\begin{abstract}There are a large number of theorems detailing the homological properties of the Stanley--Reisner ring of a simplicial complex. Here we attempt to generalize some of these results to the case of a simplicial poset. By investigating the combinatorics of certain modules associated with the face ring of a simplicial poset from a topological viewpoint, we extend some results of Miyazaki and Gr\"abe to a wider setting.
\end{abstract}


\section{Introduction}
The Stanley--Reisner ring of a simplicial complex has been one of the most useful tools in the study of combinatorial properties of simplicial polytopes, spheres, and manifolds for decades. From the upper bound theorem to the $g$-theorem, this algebraic construction has played a central role in many of the main results in this area of research. Its structure is well-understood, and the translation of geometric properties of a simplicial complex into algebraic statements about the associated ring is often clear and elegant. For an excellent overview of the historical advances allowed by these rings, the reader is referred to \cite[Sections II and III]{Stanley}.

When shifting from simplicial complexes to simplicial posets, the situation changes rather dramatically. Stanley defined the face ring of a simplicial poset as a generalization of the Stanley--Reisner ring (\cite{StanleyPosets}); in the case that the poset in question is the face lattice of a simplicial complex, the two rings are isomorphic. Unfortunately, in this broader scope the structure of the face ring becomes much less tangible. Though many classification theorems for simplicial complexes have analogs in the world of simplicial posets (see \cite[Theorem 3.10]{StanleyPosets} for an analog of the $g$-theorem, or \cite{BrowderKlee}, \cite{Masuda}, and \cite{Murai} for further abstractions), many algebraic statements about Stanley--Reisner rings do not have generalizations to face rings of simplicial posets.

The purpose of this paper is to provide such generalizations. In particular, we wish to extend the homological properties of Stanley--Reisner rings established by Gr\"abe, Hochster, Miyazaki, Reisner, and Schenzel to the world of simplicial posets and their face rings, in the combinatorial spirit of the work begun by Duval in \cite{Duval}. It is worth noting that these face rings are also squarefree modules; these objects have been studied from a various viewpoints by Yanagawa (see, e.g., \cite{Yanagawa1} and \cite{Yanagawa2}). Furthermore, similar topics have been investigated from a sheaf-theoretic viewpoint by Brun and R\"omer in \cite{BrunRomer} and by Brun, Bruns, and R\"omer in \cite{BrunBrunsRomer}.

Before stating our main results, we will fix some notation. Let $P$ be a simplicial poset with vertex set $V=\{x_1, \ldots, x_n\}$ and let $\supp(z)=\{i:x_i\preceq z\}$ denote the support of an element $z\in P$. Define $A=\field[x_1, \ldots, x_n]$, let $\mideal$ be the irrelevant ideal $(x_1, \ldots, x_n)$, and let $\mideal_{\ell}$ be the ideal $(x_1^\ell, \ldots, x_n^\ell)$. Lastly, let $A_P$ be the face ring of $P$; we consider $A_P$ as a $\ZZ^n$-graded $A$-module (we defer all definitions until the next section). Our first new result is an extension of Miyazaki's calculation of the graded pieces of the $\Ext$-modules of a Stanley--Reisner ring (\cite[Theorem 1]{Characterizations}):
\begin{reptheorem}{ExtTheorem}Let $P$ be a simplicial poset with vertex set $V$, and let $\alpha\in \ZZ^n$. Set $B=\{i:-\ell<\alpha_i<0\}$, $C=\{i:\alpha_i=-\ell\}$, and $D=\{i:\alpha_i>0\}$. If $-\ell\le \alpha_i$ for all $i$, then
	\[
	\Ext_A^i(A/\mideal_\ell, A_P)_\alpha \cong \bigoplus_{\mathclap{\substack{\supp(z)=B\cup D}}} \widetilde{H}^{i-|B|-|C|-1}\left([\hat{0}, z_D]\times\lk_P(z)_{V\smallsetminus C}\right)
	\]
	and $\Ext_A^i(A/\mideal_\ell, A_P)_\alpha=0$ otherwise. In particular, if $D\not=\emptyset$ then $\Ext_A^i(A/\mideal_{\ell}, A_P)_\alpha = 0$.
\end{reptheorem}

By considering $H_\mideal^i(A_P)$ as the direct limit of the $\Ext_A^i(A/\mideal_\ell, A_P)$ modules, we will re-prove the result of Duval (\cite[Theorem 5.9]{Duval}) that calculates the dimensions of the graded pieces of $H_\mideal^i(A_P)$. By tracing the $A$-module structure of $\Ext_A^i(A/\mideal_\ell, A_P)$ through this direct limit, we will further be able to establish an extension of Gr\"abe's result detailing the corresponding structure of $H_\mideal^i(A_P)$ (\cite[Theorem 2]{Grabe}). More precisely, our second main contribution is:

\begin{reptheorem}{CohomologyStructureTheorem}Let $\alpha=(\alpha_1, \ldots, \alpha_n)\in \ZZ^n$. Then
\[
H_\mideal^i(A_P)_\alpha\cong \bigoplus_{\mathclap{\supp(w)=\{j:\alpha_j\not=0\}}} H^{i-1}(P, \cost_P(w))
\]
if $\alpha\in \ZZ_{\le 0}^n$ and $H_\mideal^i(A_P)_\alpha = 0$ otherwise. Under these isomorphisms, the $A$-module structure of $H_\mideal^i(A_P)$ is given as follows. Let $\gamma = \alpha + \deg(x_j)$. If $\alpha_j<-1$, then $\cdot x_j: H_\mideal^i(A_P)_\alpha \to H_\mideal^i(A_P)_\gamma$ corresponds to the direct sum of identity maps
\[
\bigoplus_{\mathclap{\supp(w)=\{j:\alpha_j\not=0\}}} H^{i-1}(P, \cost_P(w))\to \bigoplus_{\mathclap{\supp(w)=\{j:\alpha_j\not=0\}}} H^{i-1}(P, \cost_P(w)).
\]
If $\alpha_j = -1$, then $\cdot x_j$ corresponds to the direct sum of maps
\[
\bigoplus_{\mathclap{\supp(w)=\{j:\alpha_j\not=0\}}} H^{i-1}(P, \cost_P(w))\to \bigoplus_{\mathclap{\supp(z)=\{j:\gamma_j\not=0\}}} H^{i-1}(P, \cost_P(z))
\]
induced by the inclusions of pairs $(P, \cost_P(w\smallsetminus\{x_j\}))\to (P, \cost_P(w))$. If $\alpha_j\ge 0$, then $\cdot x_j$ is the zero map.
\end{reptheorem}

The structure of the paper is as follows. In Section \ref{sect:preliminaries} we will review notation and some foundational results, then examine the structure of $A_P$ in preparation for later computations. In Section \ref{sect:extmodules} and Section \ref{sect:localcohomology} we will prove Theorems \ref{ExtTheorem} and \ref{CohomologyStructureTheorem}, respectively. In Section \ref{sect:characterizing} we will apply these results in a discussion of Cohen-Macaulay and Buchsbaum properties. Finally, we will close with some additional comments in Section \ref{sect:comments}.


\section{Preliminaries} \label{sect:preliminaries}


\subsection{Combinatorics and topology}
A \textbf{simplicial poset} is a partially ordered finite set $(P, \preceq)$ satisfying the following two properties:
\begin{enumerate}[(i)]
	\item There exists an element $\hat{0}\in P$ satisfying $\hat{0}\preceq y$ for all $y\in P$.
	\item For every element $y\in P$, the interval $[\hat{0}, y]$ is a Boolean lattice.
\end{enumerate}
The \textbf{rank} of an element $y$ of $P$ is the maximal length among chains from $\hat{0}$ to $y$ in $P$ and is denoted $\rk_P(y)$. When the poset $P$ is understood, we sometimes abbreviate this to $\rk(y)$. The \textbf{atoms} of $P$ are the elements of rank $1$.

If $y$ is an element of a poset $P$, then we define the \textbf{link} of $y$ in $P$ by
\[
\lk_P(y):=\{w\in P: y\preceq w\}.
\]
Note that $y$ plays the role of $\hat{0}$ in $\lk_P(y)$, and, more generally, if $w\in \lk_P(y)$ then $\rk_{\lk_P(y)}(w)=\rk_P(w)-\rk_P(y)$.
Similarly, we define the \textbf{contrastar} of $y$ in $P$ by
\[
\cost_P(y):=\{w\in P: y\npreceq w\}.
\]

If the atoms of $P$ are labeled $x_1, \ldots, x_n$, then we define the \textbf{support} of an element $y\in P$ to be the set
	\[
	\supp(y):=\{i:x_i\preceq y\}.
	\]
Given $B\subseteq \supp(y)$, we denote by $y_B$ the unique face of $P$ satisfying $y_B\preceq y$ and $\supp(y_B)=B$ (the existence and uniqueness of $y_B$ follows from the fact that $[\hat{0}, y]$ is a Boolean lattice). Similarly, if $W\subseteq \{1, \ldots, n\}$, then define
	\[
	P_W := \{y\in P: \supp(y)\subseteq W\}.
	\]
If $w$  and $y$ are two elements in $P$, then we define 
	\[
	\mub(w, y):=\{z\in P: z \text{ is a minimal upper bound for $w$ and $y$}\}.
	\]

The properties defining simplicial posets are preserved under some important operations and constructions, as outlined in the following proposition (\cite[Proposition 2.6]{Bjorner}). Recall that if $P$ and $Q$ are posets, then a partial order may be placed on the product $P\times Q$ by defining $(p_1, q_1)\preceq(p_2, q_2)$ if $p_1\preceq p_2$ in $P$ and $q_1\preceq q_2$ in $Q$.
\begin{proposition}\label{posetConstructions}Let $P$ and $Q$ be simplicial posets, and let $y\in P$. Then $P\times Q$ and $\lk_P(y)$ are both simplicial posets. If $R$ is an order ideal of $P$ (that is, $R$ is a subset of $P$ that is closed under $\preceq$), then $R$ is a simplicial poset as well; in particular, $\cost_P(y)$ is a simplicial poset.
\end{proposition}

There is a wealth of geometric meaning behind simplicial posets. In particular, every simplicial poset may be realized as the face poset of a regular CW complex (\cite[Section 3]{Bjorner}). Such a complex is much like a simplicial complex in that the cells are simplices. However, these more general objects may differ from simplicial complexes in that the intersection of two faces needs only to be a (possibly empty) subcomplex of each of the two faces, instead of a single simplex. For example, the following figure depicts the face poset of a cone formed by two triangles with two pairs of edges identified and vertices labeled $1$, $2$, and $3$, with the apex of the cone being the vertex $1$. The elements of the poset have been labeled according to their support.
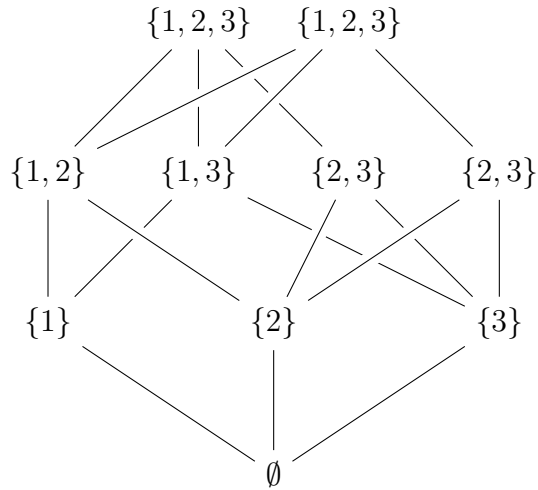
\[
\begin{tikzpicture}
\node (min) at (0, 0) {$\emptyset$};
\node (1) at (-3, 2) {$\{1\}$};
\node (2) at (0, 2) {$\{2\}$};
\node (3) at (3, 2) {$\{3\}$};
\node (12) at (-3, 4) {$\{1, 2\}$};
\node (13) at (-1, 4) {$\{1, 3\}$};
\node (231) at (1, 4) {$\{2, 3\}$};
\node (232) at (3, 4) {$\{2, 3\}$};
\node (1231) at (-1, 6) {$\{1, 2, 3\}$};
\node (1232) at (1, 6) {$\{1, 2, 3\}$};
\draw (min) -- (1) -- (12) --(1231)--(13)--(1232)--(232)--(3)--(min)--(2)--(232)
(1)--(13)--(3)--(231)--(2)--(12)--(1232)
(1231)--(231);
\draw[preaction={draw=white, -,line width=6pt}] (12)--(2)--(232);
\draw[preaction={draw=white, -,line width=6pt}] (2)--(231);
\draw[preaction={draw=white, -,line width=6pt}] (12)--(1232)--(13);
\end{tikzpicture}
\]

From this geometric viewpoint, we refer to the atoms of $P$ as \textbf{vertices} and to all elements of $P$ as \textbf{faces}. The \textbf{dimension} of a face $y\in P$ is $\dim(y):=\rk(y)-1$, and the dimension of $P$ is $\dim(P):=\max \{\dim(y):y\in P\}$. If all maximal faces of $P$ have the same dimension, then we say that $P$ is \textbf{pure}. We denote the regular CW complex associated to the poset $P$ by $|P|$, and we call it the \textbf{geometric realization} of $P$.

In addition to the geometric realization of $P$, there is another associated geometric object known as the \textbf{order complex}. First, let $\overline{P}$ denote the poset $P\smallsetminus\{\hat{0}\}$. The order complex of $\overline{P}$ is the simplicial complex $\Delta(\overline{P})$ whose elements are the chains of $\overline{P}$. In fact, $\Delta(\overline{P})$ can be realized as the face poset of the barycentric subdivision of $|P|$ (\cite[Theorem 2.1.7]{TopologyOfCW}). Furthermore, $\Delta(\overline{P\times Q})$ isomorphic as a simplicial complex to the join $\Delta(\overline{P})*\Delta(\overline{Q})$.

Let $P$ be a simplicial poset with vertices $\{x_1, \ldots, x_n\}$, and let $\field$ be a field. Denote by $\widetilde{C}^i(P)$ the vector space over $\field$ whose basis elements correspond to elements of rank $i+1$ in $P$. Next let $\mathcal{O}$ be a partial ordering of the indices $\{1, \ldots, n\}$ of the vertices of $P$ such that the restriction of $\mathcal{O}$ to the support of any one face is a total order; we call such an ordering an \textbf{orientation}. Define a coboundary operator $\delta:\widetilde{C}^i(P)\to\widetilde{C}^{i+1}(P)$ by
\[
\delta(y)=\sum_{i=1}^n(-1)^{u_i}\sum_{\mathclap{\substack{z\in \mub(x_i, y)}}} z
\]
where $u_i=|\{j\in\supp(y): i<_\mathcal{O} j\}|$. We define the \textbf{reduced cohomology groups} $\widetilde{H}^i(P)$ of $P$ with coefficients in $\field$ to be the homology of the chain complex $\widetilde{C}^\bullet(P)$ with respect to the coboundary $\delta$. If $Q$ is a non-empty order ideal in $P$, then we define the \textbf{relative cohomology groups} $H^i(P, Q)$ of the pair $(P, Q)$ with coefficients in $\field$ to be the homology of the chain complex $C^\bullet(P, Q):=\widetilde{C}^\bullet(P)/\widetilde{C}^\bullet(Q)$ with respect to the coboundary induced by $\delta$. This definition of cohomology is identical to the cellular cohomology of the CW complex $|P|$ (see, e.g., \cite[Sections 2.2 and 3.1]{Hatcher}). From this fact (along with its relative analog), the equivalence of cellular, singular, and simplicial (co)homology imply the following theorem.

\begin{theorem}
	Let $P$ be a simplicial poset, and let $Q$ be a non-empty order ideal of $P$. If $\widetilde{H}^i(|P|)$ (resp. $H^i(|P|, |Q|)$) denotes the singular cohomology of $|P|$ (resp. $(|P|, |Q|)$) and $\widetilde{H}^i(\Delta(\overline{P}))$ (resp. $H^i(\Delta(\overline{P}), \Delta(\overline{Q}))$)denotes the simplicial cohomology of $\Delta(\overline{P})$ (resp. $(\Delta(\overline{P}), \Delta(\overline{Q})$), then
	\[
	 \widetilde{H}^i(P)\cong\widetilde{H}^i(|P|)\cong\widetilde{H}^i(\Delta(\overline{P}))
	\]
	and
	\[
	H^i(P, Q)\cong H^i(|P|, |Q|)\cong 	H^i(\Delta(\overline{P}), \Delta(\overline{Q})).
	\]
\end{theorem}

\subsection{Algebra}
Let $P$ be a simplicial poset with atoms (or vertices) $x_1, \ldots, x_n$. Fix $\field$ to be an infinite field and define $A$ to be the polynomial ring over $\field$ whose variables are the atoms of $P$:
\[
A:=\field[x_1, \ldots, x_n].
\]
Next, define $\field[P]$ to be the polynomial ring whose variables are the elements of $P$:
\[
\field[P]:=\field[y:y\in P].
\]
Let $I_P$ be the ideal of $\field[P]$ defined as follows. For each pair of elements $w$ and $y$ in $P$, if $\mub(w, y)=\emptyset$ then add $wy$ as a generator of $I_P$. If $\mub(w, y)\not=\emptyset$, then add
	\[
	wy-(w\wedge y)\sum_{z\in\mub(w, y)}z
	\]
as a generator of $I_P$. Finally, add $\hat{0}-1$ to $I_P$. Define the \textbf{face ring} $A_P$ to be the quotient $\field[P]/I_P$.

Given some elements $y_1, \ldots, y_k$ in $P$ and positive integers $p_1, \ldots, p_k$, let $m$ be the monomial $\prod_{i=1}^k y_i^{p_i}$ in $A_P$. We define the \textbf{support} of $m$ to be the set
\[
\supp(m):=\{i:x_i\preceq y_j\text{ for some $j$}\}.
\]
We say that the monomial $m$ is \textbf{standard} if $y_1\succ y_2\succ\cdots \succ y_k$. In this case, we call $y_1$ the \textbf{leading variable} of $m$. Of vital importance for our arguments will be the following fact (\cite[Lemma 3.4]{StanleyPosets}).

\begin{lemma}\label{ASLLemma}
	If $P$ is a simplicial poset, then $A_P$ is an algebra with straightening law (ASL) on $P$. In particular, the standard monomials form a basis for $A_P$ as a vector space over $\field$.
\end{lemma}

Note that $A$ is naturally $\ZZ^n$-graded by setting $\deg(x_i)=\mathbf{e}_i$, the $i$-th standard basis vector of $\ZZ^n$. We may extend this grading to $\field[P]$ by defining the degree of a variable $y$ to be
\[
\deg(y):=\sum_{i\in\supp(y)}\mathbf{e}_i.
\]
Since $I_P$ is homogeneous with respect to this grading, the quotient $A_P$ is $\ZZ^n$-graded as well. We will usually view $A_P$ as a $\ZZ^n$-graded $A$-module. 

Given some degree $\alpha=(\alpha_1, \ldots, \alpha_n)$ in $\ZZ^n$, we will also speak of its \textbf{support}, defined by
\[
\supp(\alpha):=\{i:\alpha_i\not=0\}.
\]
At times it will be necessary to focus only on the ``negative'' part of $\alpha$. This negative part is the degree $\alpha^-\in\ZZ^n$ defined by
\[
\alpha_i^-:=\left\{\begin{array}{cc} \alpha_i & \alpha_i<0 \\ 0 & \alpha_i \ge 0. \end{array}\right.
\]

\subsection{The structure of $A_P$}\label{sect:structureOfAP}
Although the face ring of a simplicial poset has many subtleties and does not generally behave as well as the Stanley--Reisner ring of a simplicial complex, it does still have some nice properties. We will now highlight two of these properties with the following lemmas that will be critical to later computations. 

\begin{lemma}\label{freshmansdream}
	Let $P$ be a simplicial poset, let $w, y\in P$, and let $q$ be a positive integer. Then in $A_P$,
	\[
	(wy)^q = (w\wedge y)^q\sum_{\mathclap{z\in\mub(w, y)}}z^q.
	\]
\end{lemma}
\begin{proof}
	Set $\mub(w, y)=\{z_1, \ldots, z_k\}$. If $q=1$ or $k\le 1$, then the statement is immediate. If $k>1$ and $q>1$, then write
	\begin{align*}
	(wy)^q &= (w\wedge y)^q\left(\sum_{i=1}^k z_i\right)^q\\
	&=(w\wedge y)^q\sum_{i=1}^k z_i^q+(w\wedge y)^q\sum_\ell f_\ell
	\end{align*}
	for some monomials $f_\ell$, where each $f_\ell$ contains at least one product $z_iz_j$ with $z_i\not=z_j$. Suppose that $z_i$ and $z_j$ have some upper bound $b$ in $P$. Then $[\hat{0}, b]$ would be a Boolean lattice containing $w$ and $y$, in which $w$ and $y$ have at least two distinct minimal upper bounds. This is a contradiction, so $z_iz_j=0$ in $A_P$ and each $f_\ell=0$ as well.
\end{proof}

Our next lemma shows that the product of a standard monomial and a variable corresponding to a single vertex in $A_P$ may be expressed as a sum of ``nice'' standard monomials.

\begin{lemma}\label{prodOfMonomials}
	If $m=\prod_{i=1}^ky_i^{p_i}$ is a standard monomial in $A_p$ with $y_1\succ y_2\succ\cdots\succ y_k$ and $\ell$ is a positive integer, then $(x_j)^\ell m$ can be written in $A_p$ as 
	\[
	(x_j)^\ell m=\sum_{z\in\mub(x_j, y_1)}m_z
	\]
	where $m_z$ is a standard monomial with leading variable $z$.
\end{lemma}
\begin{proof} Write $q_r=\sum_{i=1}^rp_i$ for $0\le r\le k$. We have three cases.
	
	\textit{Case 1: $j\not\in \supp(y_1)$ and $\ell \ge q_k$.} We can re-write $x_j^\ell m$ as
	\begin{equation}\label{xjlm}
	x_j^\ell m=\left(\prod_{i=1}^k(x_jy_i)^{p_i}\right)x_j^{\ell-q_k}.
	\end{equation}
	By the definition of $A_P$, 
	\[
	x_jy_1 = (x_j\wedge y_1)\sum_{\mathclap{z\in\mub(x_j, y_1)}}z\,\,\,=\,\,\,\sum_{\mathclap{z\in\mub(x_j, y_1)}}z.
	\]
	since $j\not\in\supp(y_1)$. If $\mub(x_j, y_1)=\emptyset$, then the entire expression (\ref{xjlm}) is zero in $A_P$. If $\mub(x_j, y_1)\not=\emptyset$, then for each $z\in \mub(x_j, y_1)$ let $z_i$ be the unique element of $[\hat{0}, z]$ with vertices consisting of those in $y_i$ along with $x_j$ (note $z=z_1$). By Lemma \ref{freshmansdream}, we have
	\[
	x_j^\ell m=\sum_{\mathclap{z\in \mub(x_j, y_1)}}\,\,\,\,\,\left(x_j^{\ell-q_k}\prod_{i=1}^kz_i^{p_i}\right)+\sum_{\mathclap{z\in \mub(x_j, y_1)}}\,\,\,\,\,\left(z^{p_1}x_j^{\ell-q_k}\prod_{i=2}^k\left(\hspace{25pt}\sum_{\mathclap{\substack{w_i\in \mub(x_j, y_i) \\  w_i\not= z_i}}}w_i^{p_i}\hspace{5pt}\right)\right).
	\]
	However, all terms in the second summand on the right are zero in $A_P$. Indeed, if $z$ and some $w_i\not= z_i$ were to have some upper bound $z'$, then both $z_i$ and $w_i$ would be a minimal upper bound for $x_j$ and $y_i$ in the boolean lattice $[\hat{0}, z']$, a contradiction.
	
	Writing $m_z$ for the standard monomial $z^{p_1}\left(\prod_{i=2}^kz_i^{p_i}\right)x_j^{\ell-q_k}$, we see that 
	\[
	(x_j)^\ell m = \sum_{z\in\mub(x_j, y_1)}m_z
	\]
	as desired.

	\textit{Case 2: $j\not\in \supp(y_1)$ and $\ell<q_k$.} Let $r$ be such that $q_r \le \ell$ while $q_{r+1} >\ell$. Now re-write $x_j^\ell m$ as
	\[
	x_j^\ell m=\left(\prod_{i=1}^{r}(x_jy_i)^{p_i}\right)(x_jy_{r+1})^{\ell - q_r}y_{r+1}^{p_{r+1}-(l-q_r)}\left(\prod_{i=r+2}^ky_i^{p_i}\right).
	\]
	After replacing the products $x_jy_i$ using the definition of $A_P$ and using the notation of Case 1, our ``cross terms'' cancel again and this simplifies to
	\[
	x_j^\ell m = \sum_{\mathclap{z\in\mub(x_j, y)}}\,\,\,\,\left(\left(\prod_{i=1}^rz_i^{p_i}\right)z_{r+1}^{\ell-q_r}y_{r+1}^{p_{r+1}-(\ell-q_r)}\left(\,\,\prod_{i=r+2}^ky_i^{p_i}\right)\right)
	\]
	Now if
	\[
	m_z=z^{p_1}\left(\prod_{i=2}^rz_i^{p_i}\right)z_{r+1}^{\ell-q_r}y_{r+1}^{p_{r+1}-(\ell-q_r)}\left(\,\,\prod_{i=r+2}^ky_i^{p_i}\right)
	\]
	then again we have
	\[
	(x_j)^\ell m = \sum_{z\in\mub(x_j, y_1)}m_z.
	\]
	
	\textit{Case 3: $j\in\supp(y_1)$.} If $j\in\supp(y_i)$ for all $i$, then $m(x_j)^\ell$ is already a standard monomial and we are done. So, suppose that $j\in\supp(y_i)$ for $1\le i\le r$ while $j\not\in\supp(y_i)$ for $i>r$. Then by Cases 1 and 2,
	\[
	(x_j^\ell)\prod_{i=r+1}^ky_i^{p_i}=\sum_{z\in\mub(x_j, y_{r+1})}m_z
	\]
	for some standard monomials $m_z$ with leading coefficient $z$. Note that there is one unique $z\in \mub(x_j, y_{r+1})$ such that $z\prec y_r$. Now write
	\[
	(x_j^\ell) m= \left(\prod_{i=1}^ry_i^{p_i}\right)m_z+\left(\prod_{i=1}^ry_i^{p_i}\right)\left(\sum_{\substack{z'\in\mub(x_j, y_{r+1}) \\ z'\not=z}}m_{z'}\right).
	\]
	The first term is a standard monomial with leading coefficient $y_1$, the unique element of $\mub(x_j, y)$, while all other terms are zero as in Case 1.
\end{proof}

\section{$\Ext$ modules}\label{sect:extmodules}

\subsection{Koszul complexes and an important basis}\label{sect:basis}
All of our results will crucially depend upon calculating the $A$-modules $\Ext_A^i(A/\mideal_{\ell}, A_P)$, where $\mideal_{\ell}$ is the ideal $(x_1^\ell, \ldots, x_n^\ell)$. Let $K_\bullet^\ell$ denote the Koszul complex of $A$ with respect to the sequence $(x_1^\ell, \ldots, x_n^\ell)$ (for the construction of $K_\bullet^\ell$ and some of its associated properties, the reader is referred to \cite[Section A.3]{MonomialIdeals}). We will view each $K_t^\ell$ as the direct sum
\[
K_t^\ell = \bigoplus_{1\le i_1<i_2<\cdots<i_t\le n}A(x_{i_1}^{\ell}\wedge x_{i_2}^{\ell}\wedge\cdots\wedge x_{i_t}^{\ell}).
\]
That is, a free $A$-module basis for $K_t^\ell$ is indexed by sets $F=\{i_1, \ldots, i_t\}$ with $1\le i_1<i_2<\cdots <i_t\le n$. Given such an $F$, define $\hat{x}_F^\ell$ to be the corresponding basis element $x_{i_1}^\ell\wedge x_{i_2}^\ell\wedge\cdots\wedge x_{i_t}^\ell$.

Since $K_\bullet^\ell$ provides a projective resolution for $A/\mideal_{\ell}$, we can calculate $\Ext_A^i(A/\mideal_{\ell}, A_P)$ as the homology of the chain complex $\Hom_A(K_\bullet^\ell, A_P)$. Using the $\ZZ^n$-grading of $A$ and $A_P$, this complex splits into a direct sum of chain complexes of the form $\Hom_A(K_\bullet^\ell, A_P)_\alpha$, where $\alpha\in\ZZ^n$ (it is easily verified that the differential preserves the graded pieces). Recall (Lemma \ref{ASLLemma}) that the standard monomials form a basis for $A_P$ as a vector space over $\field$. Then $\Hom_A(K_t^\ell, A_P)_{\alpha}$ has a basis consisting of homomorphisms $f_{F, m}$ defined by
\[
f_{F, m}(\hat{x}_G^\ell)=\left\{\begin{array}{c} m, \text{ if }  G = F, \\ 0, \text{ if }  G\not=F,\end{array}\right.
\]
where $m$ is a standard monomial that satisfies
\[
\deg(m)=\deg(\hat{x}_F^\ell)+\alpha.
\]

\subsection{The main theorem}
Having examined the structure of $A_P$ as well as that of $\Hom_A(K_\bullet^\ell, A_P)$, we are now prepared to prove the first of our main results: a topological description of the graded pieces of the $A$-modules $\Ext_A^i(A/\mideal_{\ell}, A_P)$.

\begin{theorem}\label{ExtTheorem}Let $P$ be a simplicial poset with vertex set $V$and let $\alpha\in \ZZ^n$. Set $B=\{i:-\ell<\alpha_i<0\}$, $C=\{i:\alpha_i=-\ell\}$, and $D=\{i:\alpha_i>0\}$. If $-\ell\le \alpha_i$ for all $i$, then
	\[
	\Ext_A^i(A/\mideal_\ell, A_P)_\alpha \cong \bigoplus_{\mathclap{\substack{\supp(z)=B\cup D}}} \widetilde{H}^{i-|\supp(\alpha^-)|-1}\left([\hat{0}, z_D]\times\lk_P(z)_{V\smallsetminus C}\right)
	\]
	and $\Ext_A^i(A/\mideal_\ell, A_P)_\alpha=0$ otherwise. In particular, if $D\not=\emptyset$ then $\Ext_A^i(A/\mideal_{\ell}, A_P)_\alpha = 0$.
\end{theorem}

\begin{proof} Fix $\alpha$, and without loss of generality assume that
	\begin{equation}\label{assumptions}
	D = \{1, \ldots, r\}\hspace{15pt}\text{and}\hspace{15pt}C=\{s, \ldots, n\}
	\end{equation}
	for some $r$ and $s$. Let $K_\bullet^\ell$ denote the Koszul complex of $A$ with respect to the sequence $(x_1^\ell, \ldots, x_n^\ell)$. As in Section \ref{sect:basis}, we will compute $\Ext_A^i(A/\mideal_{\ell}, A_P)_\alpha$ by computing the homology of the chain complex $\Hom_A(K_\bullet^\ell, A_P)_\alpha$.
	
	Suppose first that $\alpha_i<-\ell$ for some $i$. Then there cannot exist $F=\{i_1, \ldots, i_t\}$ and a standard monomial $m$ satisfying
	\[
	\deg(m)=\deg(\hat{x}_F^\ell)+\alpha.
	\]
	That is, $\Hom_A(K_t^\ell, A_P)_\alpha=0$ and hence $\Ext_A^t(A/\mideal^\ell, A_P)_\alpha=0$. Assume from now on that $\alpha_i\ge -\ell$ for all $i$.

	If $f_{F, m}$ is a basis element of $\Hom_A(K^\ell_t, A_P)_\alpha$, let $y$ be the leading variable of $m$ and set $w=y_{D\cap F}$. Then we can consider $f_{F, m}$ as an element of the poset
	\[
	[\hat{0}, y_D]\times\lk_P(y_{B\cup D})_{V\smallsetminus C}
	\]
	where $f_{F, m}$ corresponds to the element $(w, y)$, since $B\cup D\subset \supp(y)$ and $\supp(y)\cap C = \emptyset$. Note that the rank of $(w, y)$ in this poset is given by
	\begin{align*}
		\rk(w, y)&=\rk_{[\hat{0}, y_D]}(w)+\rk_{\lk_P(y_{B\cup D})}(y)\\
		&=|\supp(w)|+|\supp(y)|-|B|-|D|\\
		&=|D\cap F|+|F\smallsetminus C|+|D\smallsetminus F|-|B|-|D|\\
		&=|D\cap F|+t-|\supp(\alpha^-)|-|D\cap F|\\
		&=t-|\supp(\alpha^-)|,
	\end{align*}
	where the fourth line follows because $|F\smallsetminus C|-|B|=t-|\supp(\alpha^-)|$ (recall that $|F|=t$) and $|D\smallsetminus F|-|D|=-|D\cap F|$. So, we can define a map
	\[
	\varphi: \Hom_A(K^\ell_t, A_P)_\alpha\to \bigoplus_{\mathclap{\substack{\supp(z)=B\cup D}}} \widetilde{C}^{t-|\supp(\alpha^-)|-1}\left(\left[\hat{0}, z_D\right]\times\lk_P(z)_{V\smallsetminus C}\right)
	\]
	by $\varphi(f_{F, m})=(w, y)$ under the correspondence above.
	
	To prove that $\varphi$ is actually an isomorphism of vector spaces, we would like to construct an inverse. Given some fixed $z$ with $\supp(z)=B\cup D$, let $(w, y)$ be an element of $[0, z_D]\times\lk_P(z)_{V\smallsetminus C}$ of rank $t-\supp(\alpha^-)$. First, set
	\[
		F = C\cup\left(\supp(y)\smallsetminus (D\smallsetminus \supp(w))\right).
	\]
	Then it is immediate that
	\[
		y_{F\cap D}=y_{\supp(w)}=w.
	\]
	Furthermore, since $\rk(w, y)=t-|\supp(\alpha^-)|=t-|B|-|C|$, we have
	\begin{align*}
		\rk(w, y)&=\rk_{[\hat{0}, y_D]}(w)+\rk_{\lk_P(y_{B\cup D})}(y)\\
		t-|B|-|C|&=|\supp(w)|+|\supp(y)|-|B|-|D|\\
		t&=|C|+|\supp(y)|-(|D|-|\supp(w)|),
	\end{align*}
	from which it follows that $|F|=t$. So, if we can construct a monomial $m$ such that the leading variable of $m$ is $y$ while $\deg(m)=\deg(\hat{x}_F^\ell)+\alpha$, then we will have constructed the desired inverse.
	
	Let $\deg(\hat{x}_F^\ell)+\alpha = \delta = (\delta_1, \ldots, \delta_n)$, and define $p_0=\min\{\delta_i:\delta_i>0\}$ and $S_0=\{i: \delta_i\ge p_0\}$. If $j>0$ and $p_{j-1}\not=\max\{\delta_i:1\le i\le n\}$, then define
	\[
	p_j = \min\{\delta_i:\delta_i>p_{j-1}\} \hspace{10pt}\text{ and } \hspace{10pt} S_j = \{i:\delta_i \ge p_j\}.
	\]
	Eventually, this process terminates at some point and $p_\ell=\max\{\delta_i:1\le i\le n\}$ for some $\ell\ge 0$. Now define
	\[
	m=y^{p_0}\prod_{i=1}^\ell y_{S_i}^{p_i-p_{i-1}}.
	\]
	We will verify that $\deg(m)=\delta$. Fix $i$. Then $\delta_i=p_j$ for some $j$, which means that $i\in S_k$ for $k\le j$ and $i\not\in S_k$ for $k> j$. Then $x_i\preceq y_{S_k}$ for $k\le j$, while $x_i\not\preceq y_{S_k}$ for $k> j$, so $\deg(y_{S_k})_i=1$ for $k\le j$ and $\deg(y_{S_k})=0$ for $k>j$. Hence,
	\[
	\deg(m)_i=p_0\deg(y)_i+\sum_{m=1}^\ell(p_m-p_{m-1})\deg(y_{S_m})_i=p_0+\sum_{m=1}^j(p_m-p_{m-1})=p_j.
	\]
	We have now shown that $f_{F, m}\in \Hom_A(K^\ell_t, A_P)_\alpha$, and it is clear by construction that $\varphi(f_{F, m})=(w, y)$. Hence, $\varphi$ is a vector space isomorphism.
	
	It remains to show that $\varphi$ provides an isomorphism of chain complexes by verifying that the differentials in each complex are preserved under $\varphi$.	Let $f_{F, m}$ be a basis element in $\Hom_A(K^\ell_\bullet, A_P)_\alpha$ with $m=y^p\prod_{i=1}^ky_i^{p_i}$. The differential on $\Hom_A(K^\ell_\bullet, A_P)_\alpha$ sends the homomorphism $f_{F, m}$ to a homomorphism $df_{F, m}$ that maps $\hat{x}_{F\cup\{j\}}\in K_{t+1}^\ell$ to $(-1)^{u_j}x_j^\ell m\in A_P$ for any $j\not\in F$, where $u_j = |\{i\in F:i<j\}|$. By Lemma \ref{prodOfMonomials}, we can write each $(-1)^{u_j}x_j^\ell m$ as
	\[
	(-1)^{u_j}x_j^\ell m = (-1)^{u_j}\sum_{\mathclap{z\in\mub(x_j, y)}}\,\,\,\, m_z,
	\]
	where $m_z$ is a standard monomial with leading variable $z$. Hence, each value of $j$ contributes to $df_{F, m}$ a term of the form
	\[
	\left((-1)^{u_j}\sum_{\mathclap{z\in \mub(x_j, y)}}f_{F\cup\{j\}, m_z}\right),
	\]
	where $m_z$ is a standard monomial with leading variable $z$. Then we may write $d(f_{F, m})$ as
	\begin{align*}
	d(f_{F, m})&=\sum_{j\not\in F}\left((-1)^{u_j}\sum_{\mathclap{z\in \mub(x_j, y)}}f_{F\cup\{j\}, m_z}\right) \\
	&=\sum_{j\in (D\smallsetminus F)}\left((-1)^{u_j}\sum_{\mathclap{z\in \mub(x_j, y)}}f_{F\cup\{j\}, m_z}\right)+\sum_{j\not\in (D\cup F)}\left((-1)^{u_j}\sum_{\mathclap{z\in \mub(x_j, y)}}f_{F\cup\{j\}, m_z}\right).
	\end{align*}
	Choose $z\in \mub(x_j, y)$. If $j\in D$, then $x_j\preceq y$ already because $D\subset \supp(y)$ and thus $z=y$. Furthermore, $z_{D\cap F}=y_{D\cap F}$ regardless of the value of $j$ and $z_{(D\cap F)\cup\{j\}}=y_{(D\cap F)\cup\{j\}}$ if $j\in D$. Hence, applying $\varphi$ to this expression allows for the following simplifications:
	\begin{align}\label{differentialfirst}
	(\varphi\circ d)(f_{F, m})&=\sum_{j\in (D\smallsetminus F)}\left((-1)^{u_j}\sum_{\mathclap{z\in\mub(x_j, y)}}(z_{(D\cap F)\cup\{j\}}, z)\right)+\sum_{j\not\in (D\cup F)}\left((-1)^{u_j}\sum_{\mathclap{z\in\mub(x_j, y)}}(z_{D\cap F}, z)\right)\nonumber \\
	&= \sum_{j\in (D\smallsetminus F)}(-1)^{u_j}(y_{(D\cap F)\cup\{j\}}, y)+\sum_{j\not\in (D\cup F)}\left((-1)^{u_j}\sum_{\mathclap{z\in\mub(x_j, y)}}(y_{D\cap F}, z)\right).
	\end{align}

	Our final step is to show that the differential on $(y_{D\cap F}, y)$ in $\widetilde{C}^\bullet([\hat{0}, y_D]\times\lk_P(y_{B\cup D})_{V\smallsetminus C})$ acts the same way. First note that if $w$ is an atom of $\lk_P(y_{B\cup D})_{V\smallsetminus C}$, then $w$ is an element of $\mub(x_j, y_{B\cup D})$ for some $j\not\in C$. Arbitrarily label the elements of $\mub(x_j, y_{B\cup D})$ by $x_{j^1}, x_{j^2}, \ldots, x_{j^{|\mub(x_j, y_{B\cup D})|}}$ for each $j$. Then by definition, we can write
	\begin{align}\label{mapfirst}
	(d\circ\varphi)(f_{F, m})&=\sum_{\mathclap{{j\in (D\smallsetminus F)}}}(-1)^{v_j}(y_{(D\cap F)\cup\{j\}}, y)+\sum_{\mathclap{\substack{j\not\in\supp(y)\\ j\not \in C}}}\,\,\,\,\,\,\,\sum_{k=1}^{\mathclap{{|\mub(x_j, y_{B\cup D})|}}}\,\,\,\,\,\,\,\,\left((-1)^{w_{j^k}}\sum_{\mathclap{z\in \mub(x_{j^k}, y)}}(y_{D\cap F}, z)\right).
	\end{align}
	where $v_j$ and $w_{j^i}$ depend upon some choice of orientations for $[\hat{0}, y_D]$ and $\lk_P(y_{B\cup D})_{V\smallsetminus C}$, respectively. Since $\supp(y)\cup C = D \cup F$, the right-hand sums in (\ref{differentialfirst}) and (\ref{mapfirst}) run over the same values of $j$. On the other hand,
	\[
	\bigcup_{j=1}^n M(x_j, y)=\bigcup_{j=1}^n\left(\bigcup_{k=1}^{|\mub(x_j, y)|} M(x_{j^k}, y)\right),
	\]
	so if the characteristic of $\field$ is $2$ then we are done. The construction of consistent orientations for fields of other characteristics is tedious, but unavoidable. If $\mathcal{O}_1$ is an ordering of the vertices in $[\hat{0}, y_D]$ and $\mathcal{O}_2$ is an ordering of the vertices in $\lk_P(y_{B\cup D})_{V\smallsetminus C}$, then
	\[
	v_j = |\{i\in \supp(y_{D\cap F}):i<_{\mathcal{O}_1}j\}|
	\]
	and
	\[
	w_{j^k}=|\{i\in \supp(y): i<_{\mathcal{O}_2} j^k\}|.
	\]
	In $[\hat{0}, y_D]$ we will choose $\mathcal{O}_1$ to be the natural ordering of the vertices. To show that $u_j=v_j$ in the left sums of (\ref{differentialfirst}) and (\ref{mapfirst}), first choose $j\in D\smallsetminus F$. If $i\in D\cap F$ satisfies $i<_{\mathcal{O}_1} j$, then $i<j$ in the natural order and $i\in F$. On the other hand, if $i\in F$ satisfies $i<j$, then $i\in D$ by our assumption (\ref{assumptions}) because $j\in D$. Hence, $u_j=v_j$ as desired.

	We now wish to construct an ordering of the $j^i$'s such that $w_{j^i}=w_{j^{k}}$ for any pair $i, k$ and such that $u_j=w_{j^i}$ for all $j\not\in(D\cup F)$. Once this is done, we will have that 
	\[
	\sum_{\substack{j\not\in\supp(y)\\ j\not \in C}}\sum_{i=1}^{|\mub(x_j, y)|}\left((-1)^{w_{j^i}}\sum_{\mathclap{z\in \mub(x_{j^i}, y)}}(y_{D\cap F}, z)\right)=
	\sum_{\substack{j\not\in\supp(y)\\ j\not \in C}}\left((-1)^{u_{j}}\sum_{\mathclap{z\in \mub(x_{j}, y)}}(y_{D\cap F}, z)\right)
	\]
	as desired. We choose the partial ordering $\mathcal{O}_2$ of the vertices of $\lk_P(y_{B\cup D})_{V\smallsetminus C}$ as follows. First, order the indices of the the vertices of $V\smallsetminus C$ as
	\begin{equation*}
	r+1<r+2<\cdots<s-1<1<2<\cdots<n,
	\end{equation*}
	and call this order $\mathcal{O}_3$. Next order the vertices of $\lk_Py$ such that $(j_1)^{i_1}<(j_2)^{i_2}$ if $j_1<j_2$ under the order $\mathcal{O}_3$. In $P$, we have $\supp(y)=\{j_1, j_2, \ldots, j_{|\supp(y)|}\}$ for some $j_i$'s with $j_i\in (V\smallsetminus C)$. Then in $\lk_P(y_{B\cup D})$, we have $\supp(y)=\{j_1^{i_1}, j_2^{i_2}, \ldots, j_{|\supp(y)|}^{i_{|\supp(y)|}}\}$ for some $j_k^{i_k}$'s with $j_k\in(V\smallsetminus C)$. In particular, no pair $j_{k_1}^{i_{k_1}}, j_{k_2}^{i_{k_2}}\in\supp(y)$ has $k_1=k_2$. Hence, the restriction of $\mathcal{O}_2$ to any face of $\lk_P(y_{B\cup D})_{V\smallsetminus C}$ is a linear order and $\mathcal{O}_2$ provides a valid orientation for $\lk_P(y_{B\cup D})_{V\smallsetminus C}$. Under this orientation,
	\[
	w_{j^i}=|\{k^m\in\supp(y): k^m<_{\mathcal{O}_2}j^i\}|=|\{k^m\in \supp(y):k<_{\mathcal{O}_3} j\}|.
	\]

	Now choose $j\not\in D\cup F$. If $i\in F$ is such that $i<j$, then $i$ cannot be in $C$ on account of our assumption (\ref{assumptions}) and thus $i\in \supp(y)$ in $P$. Then in $\lk_P(y_{B\cup D})_{V\smallsetminus C}$, $i^k\in\supp(y)$ for some single value $k$. Furthermore, $i<_{\mathcal{O}_3} j$, so $i^k<_{\mathcal{O}_2} j^m$ for all $m$. That is, $u_j\le w_{j^m}$ for all $m$.
	
	On the other hand, if $i^k\in \supp(y)$ (in $\lk_P(y_{B\cup D})_{V\smallsetminus C}$) is such that $i^k<_{\mathcal{O}_2}j^m$, then $i<_{\mathcal{O}_3} j$ and hence $i$ cannot be in $D$. Then $i\in F$, and $i<j$ in the natural order. Hence, $w_j^m\le u_j$ for all $m$, so $w_j^m=u_j$ for all $j$ and $m$ and $\varphi$ provides an isomorphism of chain complexes.

	Lastly, if $D\not=\emptyset$, then the CW complex corresponding to $[\hat{0}, z_D]\times \lk_P(z)_{V\smallsetminus C}$ is the join of a simplex with the CW complex corresponding to $\lk_P(z)_{V\smallsetminus C}$ (\cite[Section 3]{Bjorner}). In particular, it is contractible; hence, $\Ext_A^i(A/\mideal_{\ell}, A_P)_\alpha = 0$.
\end{proof}

\section{Local cohomology modules}\label{sect:localcohomology}

\subsection{Reproving Duval's result}

Let $\mideal=(x_1, \ldots, x_n)$ denote the irrelevant ideal of $A$. From Theorem \ref{ExtTheorem} we can quickly obtain the dimensions of the graded pieces of the local cohomology modules $H_\mideal^i(A_P)$, a calculation originally completed by Duval in \cite[Theorem 5.9]{Duval}. Here we will reprove the result using the techniques employed by Miyazaki in the simplicial complex case in \cite{Characterizations}, allowing for an important corollary.

\begin{theorem}\label{localCohomology}
	If $P$ is a simplicial poset and $\alpha\in \ZZ^n$ is such that $\alpha_i\le 0$ for all $i$ then
	\begin{equation}
	H_\mideal^i(A_P)_\alpha\cong \bigoplus_{\mathclap{\supp(w)=\supp(\alpha)}} \widetilde{H}^{i-|\supp(\alpha)|-1}\left(\lk_P(w)\right),
	\end{equation}
	and $H_\mideal^i(A_P)_\alpha=0$ otherwise.
\end{theorem}

\begin{proof}
	Since the chain of ideals $\mideal, \mideal_2, \mideal_3, \ldots$ is cofinal with the chain $\mideal, \mideal^2, \mideal^3, \ldots$ (that is, for all $p$, there exists $q$ and $r$ such that $\mideal^r\subseteq \mideal_p$ and $\mideal_s\subseteq \mideal^p$), we can compute $H_\mideal^i(A_P)_\alpha$ as the direct limit of the modules $\Ext_A^i(A/\mideal_{\ell}, A_P)_\alpha$ (see \cite[Chapter 7]{24Hours}). Since $\Ext_A^i(A/\mideal_{\ell}, A_P)_\alpha=0$ for all $\ell$ if $\alpha_i>0$ for some $i$, it follows that $H_\mideal^i(A_P)_\alpha=0$ if $\alpha_i>0$ for some $i$.
	
	Suppose now that $\alpha_i\le 0$ for all $i$, and choose $\ell$ such that $1-\ell\le \alpha_i$ for all $i$. Then it suffices to show that the following diagram commutes, in which the diagonal maps are the isomorphisms of Theorem \ref{ExtTheorem} and $\widetilde{\pi}$ is induced by the canonical projection $\pi:A/\mideal_{\ell+1}\to A/\mideal_\ell$:
	\[
	\begin{tikzcd}
	\Ext_A^i(A/\mideal_\ell, A_P)_\alpha \arrow[rr, "\widetilde{\pi}"] \arrow[rd] & & \Ext_A^i(A/\mideal_{\ell+1}, A_P)_\alpha \arrow[ld]\\
	&\bigoplus_{\mathclap{\substack{\, \\ \, \\ \supp(w)=\supp(\alpha)}}} \widetilde{H}^{i-|\supp(\alpha)|-1}\left(\lk_P(w)\right).
	\end{tikzcd}
	\]
	To begin, define a map $\hat{\pi}:K^{\ell+1}_\bullet\to K^{\ell}_\bullet$ by
	\[
	\hat{\pi}(x_{i_1}^{\ell+1}\wedge\cdots\wedge x_{i_t}^{\ell+1})=(x_{i_1}\cdots x_{i_t})(x_{i_1}^\ell\wedge\cdots\wedge x_{i_t}^\ell).
	\]
	Then the following diagram of projective resolutions commutes
	\[
	\begin{tikzcd}
	K^{\ell+1}_{\bullet} \arrow[r] \arrow[d, "\hat{\pi}"] & A/\mideal_{\ell+1}\arrow[d, "\pi"] \arrow[r] & 0\\
	K^\ell_\bullet \arrow[r] & A/\mideal_\ell \arrow[r] & 0,
	\end{tikzcd}
	\]
	so dualizing produces another commutative diagram:
	\[
	\begin{tikzcd}
	\Hom_A(K^{\ell+1}_{\bullet}, A_P)_\alpha & \Hom_A(A/\mideal_{\ell+1}, A_P)_\alpha   \arrow[l] & 0  \arrow[l]\\
	\Hom_A(K^\ell_\bullet, A_P)_\alpha  \arrow[u, "\hat{\pi}^*"] & \Hom_A(A/\mideal_\ell, A_P)_\alpha \arrow[l] \arrow[u, "\pi^*"] & 0. \arrow[l]
	\end{tikzcd}
	\]
	When taking homology to compute $\Ext_A^i (A/\mideal_\ell, A_P)_\alpha$ and $\Ext_A^i (A/\mideal_{\ell+1}, A_P)_\alpha$, the map $\hat{\pi}^*$ may be used to compute $\widetilde{\pi}$. That is, we only need to check that the following diagram commutes:
	\begin{equation}\label{commutativeChainDiagram}
	\begin{tikzcd}
	\Hom_A(K^\ell_t, A_P)_\alpha \arrow[rr, "\hat{\pi}^*"] \arrow[rd] & & \Hom_A(K_t^{\ell+1}, A_P)_\alpha \arrow[ld]\\
	&\bigoplus_{\mathclap{\substack{\, \\ \, \\ \supp(w)=\supp(\alpha)}}} \widetilde{C}^{t-|\supp(\alpha)|-1}\left(\lk_P(w)\right).
	\end{tikzcd}
	\end{equation}
	Let $f_{F, m}$ be a basis element for $\Hom_A(K_t^\ell, A_P)_\alpha$, where $F=\{i_1, \ldots, i_t\}$ and $m=y^p\prod_{i=1}^k y_i^{p_i}$ with $y\succ y_1\succ y_2\succ\cdots\succ y_k$. Then $f_{F, m}$ corresponds to the face $y$ in $\widetilde{C}^{t-|\supp(\alpha)|-1}(\lk_P(y_{\supp(\alpha)}))$ under the map $\varphi$ of Theorem \ref{ExtTheorem}.
	
	On the other hand, $f_{F, m}\circ\hat{\pi}$ is the homomorphism that sends $\hat{x}_F^{\ell+1}$ to $(x_{i_1}\cdots x_{i_t})m$ and all other basis elements of $K_t^{\ell+1}$ to zero. Note that $x_{i_j}\in\supp(y)$ for $1\le j \le t$ because $1-\ell\le \alpha_i\le 0$ for all $i$. By iterating Lemma \ref{prodOfMonomials} $t$ times, we have $(x_{i_1}\cdots x_{i_t})m=m'$ where $m'$ is a standard monomial in $A_P$ with leading variable $y$. Hence, $\varphi(f_{F, m})=(\varphi\circ\hat{\pi}^*)(f_{F, m})$ and Diagram (\ref{commutativeChainDiagram}) commutes.
\end{proof}

The proof of Theorem \ref{localCohomology} immediately provides the following corollary, which we will use in Section \ref{sect:characterizing} to characterize Buchsbaum simplicial posets.
\begin{corollary}\label{canonicalMapIso}
	If $\alpha\in \ZZ^n$ is such that $1-\ell\le \alpha_i$ for all $i$, then the canonical map
	\[
	\Ext_A^i(A/\mideal_\ell, A_P)_\alpha \to H_\mideal^i(A_P)_\alpha
	\]
	is an isomorphism for all $i$.
\end{corollary}

\subsection{The $A$-module structure of $H_\mideal^i(A_P)$}
The simplicial complex analog of Theorem \ref{localCohomology} counting the dimensions of the local cohomology modules $H_\mideal^i(\field[\Delta])$ of a Stanley--Reisner ring $\field[\Delta]$ was proved first by Hochster (\cite[Section II.4]{Stanley}). Later, Gr\"{a}be (\cite[Theorem 2]{Grabe}) examined the graded module structure of $H_\mideal^i(\field[\Delta])$. Here we will prove a similar result for $H_\mideal^i(A_P)$. First, we will need the following simplicial poset analog of \cite[Lemma, p.273]{Grabe}.
	\begin{lemma}\label{lkCost}
		Let $P$ be a simplicial poset and let $y\in P$. Then
		\[
		\widetilde{H}^{i-|\supp(y)|-1}(\lk_P(y))\cong H^{i-1}(P, \cost_P(y)).
		\]
	\end{lemma}
	\begin{proof}We clearly have an isomorphism of vector spaces $\widetilde{C}^{i-|\supp(y)|}(\lk_P(y))\cong C^i(P, \cost_P(y))$ under the map $\varphi:w\mapsto w$. To show that this is an isomorphism of chain complexes, we must proceed as in the proof of Theorem \ref{ExtTheorem}.
	
	Label the vertices of $P$ by $x_1, \ldots, x_n$, and let $\mathcal{O}_1$ be any ordering of the vertices of $P$ in which the vertices contained in $\supp(y)$ come last. If $w$ is a vertex of $\lk_P(y)$, then $w$ is an element of $\mub(x_j, y)$ for some $j$ with $j\not\in\supp(y)$. Arbitrarily label the elements of $\mub(x_j, y)$ by $x_{j^1}, x_{j^2}, \ldots, x_{j^{|\mub(x_j, y)|}}$. Let $\mathcal{O}_2$ be the partial ordering of the vertices of $\lk_P(y)$ given by setting $j_1^{k_1}<j_2^{k_2}$ if $j_1<_{\mathcal{O}_1}j_2$.
	
	If $w\in(P, \cost_P(y))$, then write $\supp(w)=\{j_1, \ldots, j_t\}$. In $\lk_P(y)$, we have $\supp(w)=\{j_1^{k_1}, \ldots, j_t^{k_t}\}$ for some $k_i$'s. That is, $\mathcal{O}_2$ is a total order when restricted to $\supp(w)$ for any $w\in \lk_P(y)$, so $\mathcal{O}_2$ provides a valid orientation for $\lk_P(y)$.
	
	In $C^t(P, \cost_P(y))$,
	\[
	\delta(w)=\sum_{j\not\in\supp(w)}\left((-1)^{u_j}\sum_{z\in\mub(x_j, w)}z\right)
	\]
	where
	\[
	u_j=|\{i\in\supp(w):i<_{\mathcal{O}_1}j\}|.
	\]	
	On the other hand, in $\widetilde{C}^{t-|\supp(y)|}(\lk_P(y))$,
	\[
	\delta(w)=\sum_{j\not\in\supp(w)}\left(\sum_{k=1}^{|\mub(x_j, w)|}(-1)^{v_{j^k}}\sum_{\mathclap{z\in\mub(x_{j^k}, w)}}z\,\,\,\,\,\right)
	\]
	where
	\[
	v_{j^k}=|\{i\in \supp(w):i<_{\mathcal{O}_2}j^k\}|.
	\]
	Note that
	\[
	\bigcup_{j\not\in\supp(w)}M(x_j, w)=\bigcup_{j\not\in\supp(w)}\left(\bigcup_{k=1}^{|\mub(x_{j^k}, w)|}\mub(x_{j^k}, w)\right)
	\]
	and that $u_j=v_{j^k}$ for fixed $j$ and arbitrary $k$ under our choice of ordering for $\mathcal{O}_1$. The result now follows.
\end{proof}

With this lemma in hand, we are now prepared to prove our second main result.

\begin{theorem}\label{CohomologyStructureTheorem}Let $\alpha\in \ZZ^n$. Then
	\[
	H_\mideal^i(A_P)_\alpha\cong \bigoplus_{\mathclap{\supp(w)=\supp(\alpha)}} H^{i-1}(P, \cost_P(w))
	\]
	if $\alpha\in \ZZ_{\le 0}^n$ and $H_\mideal^i(A_P)_\alpha = 0$ otherwise. Under these isomorphisms, the $A$-module structure of $H_\mideal^i(A_P)$ is given as follows. Let $\gamma = \alpha + \deg(x_j)$. If $\alpha_j<-1$, then $\cdot x_j: H_\mideal^i(A_P)_\alpha \to H_\mideal^i(A_P)_\gamma$ corresponds to the direct sum of identity maps
	\[
	\bigoplus_{\mathclap{\supp(w)=\supp(\alpha)}} H^{i-1}(P, \cost_P(w))\to \bigoplus_{\mathclap{\supp(w)=\supp(\gamma)}} H^{i-1}(P, \cost_P(w)).
	\]
	If $\alpha_j = -1$, then $\cdot x_j$ corresponds to the direct sum of maps
	\[
	\bigoplus_{\mathclap{\supp(w)=\supp(\alpha)}} H^{i-1}(P, \cost_P(w))\to \bigoplus_{\mathclap{\supp(z)=\supp(\gamma)}} H^{i-1}(P, \cost_P(z))
	\]
	induced by the inclusions of pairs $(P, \cost_P(w\smallsetminus\{x_j\}))\to (P, \cost_P(w))$. If $x_j\not\in \supp(\alpha)$, then $\cdot x_j$ is the zero map.
\end{theorem}

\begin{proof}
	The first statement is immediate, as it is a reformulation of Theorem \ref{localCohomology} using the isomorphisms of Lemma \ref{lkCost}. Let $\ell$ be such that $1-\ell\le \alpha_i\le 0$ for all $i$. By Corollary \ref{canonicalMapIso}, we only need to show that the multiplication maps above are valid for $\Ext_A^t(A/\mideal_\ell, A_P)_\alpha$, as the following diagram commutes:
	\[
	\begin{tikzcd}
	\Ext_A^t(A/\mideal_\ell, A_P)_\alpha \arrow[r, "\sim"] \arrow[d, "\cdot x_j"] & H_\mideal^t(A_P)_\alpha \arrow[d, "\cdot x_j"] \\
	\Ext_A^t(A/\mideal_\ell, A_P)_\gamma \arrow[r, "\sim"] & H_\mideal^t(A_P)_\gamma.
	\end{tikzcd}
	\]
	First note that if $x_j\not\in \supp(\alpha)$, then $\gamma_j=1$ so that $\Ext_A^t(A/\mideal_\ell, A_P)_\gamma= H_\mideal^t(A_P)_\gamma=0$. Assume now that $x_j\in\supp(\alpha)$ and let $f_{F, m}$ be a basis element for $\Hom_A^t(A/\mideal_\ell, A_P)_\alpha$ in which $m=y^p\prod_{i=1}^k y_i^{p_i}$ with $y\succ y_1 \succ y_2\succ\cdots \succ y_k$.
	
	By the proof of Theorem \ref{ExtTheorem}, $f_{F, m}$ corresponds to the face $y$ of $\widetilde{C}^{t-|\supp(\alpha)|-1}(\lk_P(y_{\supp(\alpha)}))$. On the other hand, $x_j\cdot f_{F, m}=f_{F, x_jm}$ and $x_j\prec y$, so by Lemma \ref{prodOfMonomials} we can write $x_j m$ as a standard monomial $m'$ in which the leading variable of $m'$ is still $y$. Then $f_{F, m'}$ corresponds to the face $y$ of $\widetilde{C}^{t-1}(P, \cost_P(y_{\supp(\gamma)}))$.
	
	Suppose first that $\gamma_j<0$. Then $\supp(\alpha)=\supp(\gamma)$, so
	\[
		\widetilde{C}^{t-1}(P, \cost_P(y_{\supp(\alpha)}))=\widetilde{C}^{t-1}(P, \cost_P(y_{\supp(\gamma)}))
	\]
	and $\cdot x_v$ corresponds to the identity map, as desired.
	
	If $\gamma_j=0$, then we have a map
	\[
	\widetilde{C}^{t-1}(P, \cost_P(y_{\supp(\alpha)})) \xrightarrow{\cdot x_v} \widetilde{C}^{t-1}(P, \cost_P(y_{\supp(\alpha)\smallsetminus\{x_j\}}))
	\]
	taking $y$ to $y$, which is identical to the map induced by the inclusion of pairs \linebreak $(P, \cost_P(y_{\supp(\alpha)\smallsetminus\{x_j\}}))\to (P, \cost_P(y_{\supp(\alpha)}))$.
\end{proof}

\section{Characterizing Buchsbaumness}\label{sect:characterizing}
Central to the historical study of simplicial complexes have been the notions of a complex being \textbf{Cohen-Macaulay} or \textbf{Buchsbaum}. In analogy to the results of Reisner (\cite{Reisner}) and Schenzel (\cite{Schenzel}) for simplicial complexes, we will say that a poset $P$ is Cohen-Macaulay (over $\field$) if its order complex is Cohen-Macaulay (over $\field$) in the topological sense; that is,
\[
\widetilde{H}^i(\lk_{\Delta(\overline{P})}(F))=0
\]
for all faces $F\in\Delta(\overline{P})$ and all $i<\dim(\lk_{\Delta(\overline{P})}(F))$. We say that $P$ is Buchsbaum (over $\field$) if it is pure and the link of every vertex of $\Delta(\overline{P})$ is Cohen-Macaulay (over $\field$).

There are also algebraic notions of Cohen-Macaulayness and Buchsbaumness. Let $R$ be a ring with maximal ideal $\mideal$ and let $M$ be a noetherian $R$-module of dimension $d$. We call a sequence of elements $\theta_1, \ldots, \theta_d\in\mideal$ a \textbf{system of parameters} if $M/(\theta_1, \ldots, \theta_d)M$ is a finite-dimensional vector space over $\field$. If
\[
(\theta_1, \ldots, \theta_{i-1})M:\theta_i = 0
\]
for $i=1, \ldots, d$, then we say that $\theta_1, \ldots, \theta_d$ is an \textbf{$M$-sequence}. If every homogeneous system of parameters for $M$ is an $M$-sequence, then we say that $M$ is a \textbf{Cohen-Macaulay} $R$-module. If
\[
(\theta_1, \ldots, \theta_{i-1})M:\theta_i = (\theta_1, \ldots, \theta_{i-1})M:\mideal
\]
for $i=1, \ldots, d$, then we say that $\theta_1, \ldots, \theta_d$ is a \textbf{weak $M$-sequence}. If every homogeneous system of parameters for $M$ is a weak $M$-sequence, then we call $M$ a \textbf{Buchsbaum} $R$-module. We say that the ring $R$ is Cohen-Macaulay (resp. Buchsbaum) if it is Cohen-Macaulay (resp. Buchsbaum) as a module over itself. Stanley showed (\cite[Theorem 3.4, p. 314]{Polytopes}) that the topological notion of Cohen-Macaulayness is tied to this algebraic property (proofs of the forward direction appear frequently in the literature, but we have had trouble finding a proof of the reverse; however, it does follow quickly from Duval's results, e.g., \cite[Corollary 6.1]{Duval}):
\begin{theorem}
	Let $P$ be a simplicial poset. Then $P$ is Cohen-Macaulay as a poset over $\field$ if and only if $A_P$ is a Cohen-Macaulay ring.
\end{theorem}
In fact, an additional characterization is that $A_P$ is a Cohen-Macaulay $A$-module. Our next goal is obtaining an analogous characterization for the Buchsbaum property. To prove our result, we will need the following homological characterization of Buchsbaumness (\cite[Theorem I.3.7]{StVo}):
\begin{theorem}\label{surjectivityThm}
	Let $\field$ be an infinite field and let $M$ be a graded $A$-module with $\dim (M)>0$. Then $M$ is a Buchsbaum $A$-module if and only if the canonical maps
	\[
	\Ext_A^i(\field, M)\to H_\mideal^i(M)
	\]
	are surjective for $i<\dim(M)$.
\end{theorem}

Along with Theorem \ref{CohomologyStructureTheorem} and Corollary \ref{canonicalMapIso}, this result will provide a quick proof of the following generalization of \cite[Theorem II.2.4]{StVo}, the forward direction of which appeared first as \cite[Proposition 6.2]{NS-Socles}:

\begin{theorem}\label{BuchsbaumCriterionForAP}
	Let $P$ be a simplicial poset. Then $P$ is Buchsbaum as a poset if and only if $A_P$ is a Buchsbaum $A$-module.
\end{theorem}

\begin{proof}Suppose first that $P$ is a Buchsbaum poset. Under the isomorphisms
	\begin{equation}\label{linkIsoms}
	\widetilde{H}^i(\lk_{\Delta(\overline{P})}(F))\cong\widetilde{H}^i(\Delta(\overline{\lk_P(w)}))\cong \widetilde{H}^i(\lk_P(w))
	\end{equation}
	for $F$ a saturated chain in $(\hat{0}, w]$, by Theorem \ref{localCohomology} along with the purity of $P$, $H_\mideal^i(A_P)$ is concentrated in degree zero for $i<\dim(A_P)$. By Corollary \ref{canonicalMapIso}, the canonical map $\Ext_A^i(\field, A_P)_0\to H_\mideal^i(A_P)_0$ is always an isomorphism. Hence, $A_P$ satisfies the conditions of Theorem \ref{surjectivityThm} and $A_P$ is a Buchsbaum $A$-module.

	Now suppose that $A_P$ is a Buchsbaum $A$-module, and for the moment assume that $P$ is pure. By \cite[Corollary I.2.4]{StVo}, $\mideal\cdot H_\mideal^i(A_P)=0$ for all $i\not=\dim A_P$. By Theorem \ref{CohomologyStructureTheorem} and purity, this is only possible if
	\[
	\widetilde{H}^i(\lk_P(y))=0
	\]
	for all $y\not=\emptyset$ in $P$ and all $i<\dim(\lk_P(y))$. Then once more, the isomorphisms in (\ref{linkIsoms}) imply that $\widetilde{H}^i(\lk_{\Delta(\overline{P})}(F))=0$ for all faces $F\not=\emptyset$ of $\Delta(\overline{P})$ and all $i<\dim (\lk_{\Delta(\overline{P})}(F))$. In particular, $\lk_{\Delta(\overline{P})}(v)$ is Cohen-Macaulay for each vertex $v$ of $\Delta(\overline{P})$.
	
	It remains to show that $P$ is pure. Suppose to the contrary that there exists some maximal face $y\in P$ with $\dim(y)=k$, while $\dim(P)=d-1$ and $k<d-1$. Note first that $\widetilde{H}^{-1}(\lk_P(y))\not=0$.
	
	Let $\alpha\in\ZZ^n$ be such that $\supp(\alpha)=\supp(\alpha)+\mathbf{e}_i=\supp(y)$ for some $\mathbf{e}_i$. By Theorem \ref{localCohomology}, $\widetilde{H}^{-1}(\lk_P(y))$ is a submodule of $H_\mideal^{k+1}(A_P)_\alpha$. Furthermore, the multiplication map $\cdot x_i:H_\mideal^{k+1}(A_P)_\alpha\to H_\mideal^{k+1}(A_P)_{\alpha+\mathbf{e}_i}$ is the identity map on this summand by Theorem \ref{CohomologyStructureTheorem}. Then $\mideal\cdot H_\mideal^{k+1}(A_P)\not=0$ and $k+1<\dim(A_P)$, which contradicts \cite[Corollary I.2.4]{StVo}. Hence, $P$ must be pure.
\end{proof}

\section{Comments and further applications}\label{sect:comments}

Our Theorem \ref{CohomologyStructureTheorem} detailing the $A$-module structure of $H_\mideal^i(A_P)$ has the potential to extend many results about simplicial complexes relying upon Gr\"abe's theorem to the simplicial poset case. For example, many of the theorems for face rings of simplicial complexes with isolated singularities found in \cite{MNS-Sing}, \cite{NS-Sing}, and \cite{AlmostBuchs} appear likely to be true in this greater generality.

On the other hand, Theorems \ref{ExtTheorem} and \ref{CohomologyStructureTheorem} may be worth further specializing in their own right. Simplicial posets come about frequently as quotients of simplicial complexes by free group actions (see \cite[Section 6]{Garsia}). It has been shown in \cite{GroupActions} that the $\Ext$- and local cohomology modules of a simplicial complex admitting a free action by an abelian group inherit some nice symmetry properties. It would be interesting to see to what extent these results continue to hold in the simplicial poset case.

\section*{Acknowledgments}The author was partially supported by NSF grants DMS-1361423 and DMS-1664865. I am very grateful to Isabella Novik for many helpful suggestions and guidance during the writing of this paper, as well as to Satoshi Murai for additional comments. Finally, I would like to thank Tim R\"omer for pointing me toward his previous related research with Morton Brun and Winfried Bruns.

\bibliographystyle{alpha}
\bibliography{SimplicialPosets-biblio}
\end{document}